\documentclass{amsart}
\usepackage{amssymb}
\usepackage{amsmath,etoolbox}
\usepackage{graphicx}
\usepackage{float}
\restylefloat{figure}
\def\({\left(}
\def\){\right)}

\newtheorem{lema}{Lemma}[section]

\newtheorem*{teorema*}{Theorem}

\newtheorem{remark}[lema]{Remark}

\newtheorem{corollary}[lema]{Corollary}
\newtheorem{theorem}[lema]{Theorem}

\newtheorem{definition}[lema]{Definition}

  {\par \hfill \fbox{}}
  {\par \hfill \fbox{}}
  {\par \hfill }
  {\par \hfill }

\def\beq{\begin{equation}}
\def\eeq{\end{equation}}

\def\epsilon{\varepsilon}

\begin{document}

\title[Bicomplex Mergelyan's Approximation Theorems ]{
Bicomplex Mergelyan's Approximation Theorems}

\author{Amjad Ali}
\address{Department of Mathematics, University of Jammu,
Jammu 180006, INDIA.} \email{amjadladakhi687@gmail.com}

\author{Mohd arif}
\address{Department of Mathematics, University of Jammu,
Jammu 180006, INDIA.} \email{azizymaths@gmail.com}

%\author{Rajat Singh}
%\address{Department of Mathematics, University of Jammu,
%Jammu 180006, INDIA.} \email{rajat.singh.rs634@gmail.com}

\author{Romesh Kumar}
\address{Department of Mathematics, University of Jammu,
Jammu 180006, INDIA.} 
\email{romeshmath@gmail.com}
\thanks{The first author is supported by UGC in the form of JRF(Junior Research Fellowship) and the second author is supported by CSIR with JRF via Scholarship file No. 09/100(0215)/2018-EMR-I }
\date{February 2022}
%\dedicatory{}
%\commby{}

\begin{abstract} 
In this paper we prove Area theorem, Biebarbach's  Theorem,  Koebe Quarter Theorem and Mergelyan's Approximation Theorem in the bicomplex framework.
\end{abstract}

% ----------------------------------------------------------------------
%\begin{center}
\maketitle
%\end{center}
% ----------------------------------------------------------------------

\section{Introduction and Preliminaries}

The main aim of this paper is to prove the bicomplex version of some fundamental theorems form univalent function theory. Some work in this direction was initiated in \cite{Arif1} and \cite{Arif2}.

We define set of bicomplex numbers denoted by $\mathbb{BC}$ as 
$$\mathbb{BC}=\left\{Z=x_1+x_2\mathbf{i}+x_3\mathbf{j}+x_4\mathbf{k}: x_1,x_2,x_3,x_4 \in \mathbb{R}\right\}= \left\{Z=Z_1+\mathbf{j}Z_2: Z_1,Z_2 \in \mathbb{C}(\mathbf{i})\right\},$$
where $\mathbf{i}$ and $\mathbf{j}$ are two imaginary units, satisfying 
$$ \mathbf{i} \neq \mathbf{j};~~~\mathbf{ij}=\mathbf{ji} =\mathbf{k}; ~~~ \mathbf{i^2}=\mathbf{j^2}=-1. $$
Due to the fact that the set $\mathbb{BC}$ has two imaginary units i.e., $\mathbf{i}$ and $\mathbf{j}$, $\mathbb{BC}$ has  three conjugations. The bar-conjugation, $\dagger$-conjugation and  $\ast$-conjugation  defined as $\overline{Z}=\overline{Z}_{1}+\mathbf{j} \overline{Z}_{2}$,
$Z^{\dagger}=Z_1-\mathbf{j}Z_2$ and $Z^\ast=\overline{Z}^{\dagger}=\overline{Z}_{1}-\mathbf{j}\overline{Z}_{2},$ respectively, where $\overline{Z}_{1},\overline{Z}_{2}$ are the usual conjugations of complex numbers $Z_1,Z_2$ in $\mathbb{C}(\mathbf{i}).$
Accordingly, three types of  moduli  arise and these are $Z \cdot Z^{\dagger}$, $ Z \cdot \overline{Z}$ and $ Z \cdot Z^{\ast}$.
It is to be noted that  these moduli are $\mathbb{C}(\mathbf{i}),\mathbb{C}(\mathbf{j})$ and $\mathbb{D}$-valued respectively.
For details of conjugations on set of bicomplex numbers, see \cite{Alpa}, \cite{Luna6} and \cite{Pric}.
Further, if  $Z \ne 0$, $Z \cdot Z^{\dagger}=|Z|_{\mathbf{i}}^{2}=0$, then $Z$ is said to be a zero-divisor. We denote the set of all zero-divisors by
$$\mathcal{O}=\left\{Z=Z_1+\mathbf{j}Z_2: Z \ne 0,Z \cdot Z^{\dag}=Z_1^2+Z_2^2=0\right\}$$
and is called the null cone of the set of bicomplex numbers $\mathbb{BC}$. Also we denote $\mathcal{O}_0 = \mathcal{O} \cup \{0\}.$ \\
There are two special zero divisors called idempotent elements defined as
$$\mathbf{e_1} =\frac{1}{2}(1+\mathbf{k})~~~\mbox{and}~~~\mathbf{e_2}=\frac{1}{2}(1-\mathbf{k})$$
and have following properties
\begin{equation}
\label{eq:1}
\mathbf{e_1} + \mathbf{e_2} = 1;~~~~ \mathbf{e_1} -  \mathbf{e_2} = \mathbf{k};~~~~
\mathbf{e_1}\cdot\mathbf{e_2} = 0 ; ~~~~~ \mathbf{e_1}\cdot\mathbf{e_1} = \mathbf{e_1} ; ~~~~~~ \mathbf{e_2}\cdot\mathbf{e_2} = \mathbf{e_2}.
\end{equation}
We can see that $\mathbf{e_1}$ and $\mathbf{e_2}$ are zero divisors and are mutually complementary idempotent elements.
The sets $\mathbb{BC}_{\mathbf{e_1}}=\mathbf{e_1}\mathbb{BC}$ and $\mathbb{BC}_{\mathbf{e_2}}=\mathbf{e_2}\mathbb{BC}$ are (principal) ideals in the ring $\mathbb{BC}$ and have the decomposition property 
$$\mathbb{BC}_{\mathbf{e_1}}\cap \mathbb{BC}_{\mathbf{e_2}}=\{0\}$$ 
and
\begin{equation}
\mathbb{BC}=\mathbb{BC}_{\mathbf{e_1}}+\mathbb{BC}_{\mathbf{e_2}}.
\end{equation} 
Thus every bicomplex number have idempotent representation in $\mathbf{e_1}$ and $\mathbf{e_2}$ as  
$$ Z = \beta_1 \mathbf{e_1} + \beta_2 \mathbf{e_2} ,$$
where $\beta_1=Z_{1}-\mathbf{i}Z_{2}$ and $\beta_2=Z_{1}+\mathbf{i}Z_{2}$ are complex numbers.
The $\mathbb{D}$-valued norm of the bicomplex number $Z$ denoted  by $|Z|_{\mathbf{k}}$ is defined as $|Z|_{\mathbf{k}}=|\beta_1|\mathbf{e_1}+|\beta_2|\mathbf{e_2}$, where $|\beta_1|$ and $|\beta_2|$ are the  usual modulus of complex numbers $\beta_1$ and $\beta_2$. Further $|Z\cdot W|_{\mathbf{k}}=|Z|_{\mathbf{k}}\cdot|W|_{\mathbf{k}}$. For above discussions we refer to \cite{Alpa}, \cite{Luna6} and \cite{Pric}.

The $\overline{\mathbb{BC}}$ is not  one point Alexendrov compactification, but is the union of $\mathbb{BC}$ with three different types of infinities:
$$ \overline{\mathbb{BC}} = \mathbb{BC} \cup \left\{ \infty\mathbf{e_1} + \mathbb{C}(\mathbf{i})\mathbf{e_2}\right\} \cup \left\{ \mathbb{C}(\mathbf{i})\mathbf{e_1} + \infty\mathbf{e_2}\right\} \cup  \left\{ \infty\mathbf{e_1} + \infty\mathbf{e_2}\right\}. $$ 
Thus infinity in $\mathbb{BC}$ have three different  type of elements. For more details we refer to \cite{Luna8}.

A set $\Omega \subset \mathbb{BC}$ is said to be product-type set if  $\Omega$ can be written as $\Omega = \Omega_1 \mathbf{e_1} + \Omega_2 \mathbf{e_2} $ where $\Omega_1 = \Pi_{1,\mathbf{i}}(\Omega)$ and $\Omega_2 = \Pi_{2,\mathbf{i}}(\Omega)$ are the projections of $\mathbb{BC}$ on $\mathbb{C}(\mathbf{i})$.
A set $\Omega \subset \mathbb{BC} $ is said to be product-type domain in $\mathbb{BC}$  if  $\Omega_1$ and $\Omega_2 $ are  domains in the complex plane $\mathbb{C}(\mathbf{i})$. A function $F:\Omega\to \mathbb{BC}$ is said to be product-type function if there exist $F_i:\Omega_i \to \mathbb{C}(\mathbf{i})$ for $i = 1,2$ such that $F(\beta_1 \mathbf{e_1} + \beta_2 \mathbf{e_2}) = F_1(\beta_1)\mathbf{e_1} + F_2(\beta_2)\mathbf{e_2}$ for all $\beta_1 \mathbf{e_1} + \beta_2 \mathbf{e_2} \in \Omega$.  Also, if $F$ and $G$ are bicomplex product-type functions defined on $\Omega\subset \mathbb{BC}$ and $A \in \mathbb{BC}$, then $F+G$, $AF,~ \left|F\right|_{\mathbf{k}}$ are product-type and further if $H:S\subset \mathbb{BC} \to \Omega$  a product-type function, then we have $F\circ H$  also product-type. For above discussions we refer to \cite{Dave}, \cite{Luna6} and \cite{Reye}.

\begin{definition}
Let $\Omega \subset \mathbb{BC}$ be an open set. Then the mapping $F:\Omega \to \mathbb{BC}$ is called $\mathbb{BC}$-holomorphic in $\Omega$ if the limit 
$$F^{'}(Z_{o})=\underset{h=Z-Z_{0}\notin \mathcal{O}_{0}}{\underset{h \to 0}{lim}}\frac{F(Z)-F(Z_{o})}{Z-Z_{o}}$$
																	exist for every $Z \in \Omega$.   
\end{definition}

\begin{definition}
Let $\Omega=\Omega_1\mathbf{e_1}+\Omega_2\mathbf{e_2}\subseteq\mathbb{BC}$ be an open set. Then a function $F:\Omega \to \mathbb{BC}$ is said to be a $\mathbb{BC}$-univalent function on $\Omega$ if the following conditions are satisfied:\\
(a) $F$ is product-type function.\\
(b) $F(Z_{1}) \neq F(Z_{2}),~\forall~Z_{1},Z_{2} \in \Omega$  with $Z_{1} \neq Z_{2}.$
\end{definition} 

%\begin{definition}
%Let $\Omega=\Omega_1\mathbf{e_1}+\Omega_2\mathbf{e_2}\subseteq\mathbb{BC}$ be an open set. Then a function $F:\Omega \to \mathbb{BC}$ is said to be a $\mathbb{BC}$-conformal on $\Omega$ if $F$ is $\mathbb{BC}$-holomorphic as well as $\mathbb{BC}$-univalent on $\Omega$.
%\end{definition}

\begin{definition}
A unit disk in $\mathbb{BC}$ is defined as:
$$\mathbb{D}_{\mathbb{BC}}=\{Z=\beta_{1}\mathbf{e}_{1}+\beta_{2}\mathbf{e}_{2}:(\beta_{1},\beta_{2}) \in \mathbb{D} \times \mathbb{D} \},$$
where $\mathbb{D}$ be a complex valued unit disk in $\mathbb{C}(\mathbf{i})$.
\end{definition}

\begin{definition}\cite[Definition 2.3, Page-6]{Arif2}
Let $\mathcal{F}$ denote the set of $\mathbb{BC}$-holomorphic, $\mathbb{BC}$-univalent functions on the unit disk $\mathbb{D}_{\mathbb{BC}}$ normalized by the condition $F(0)=0$ and $F^{'}(0)=1$. That is,
$$\mathcal{F}=\{F:\mathbb{D}_{\mathbb{BC}} \to \mathbb{BC} :\mbox{F is $\mathbb{BC}$-holomorphic and $\mathbb{BC}$-univalent on}~ \mathbb{D}_{\mathbb{BC}}, F(0)=0, F^{'}(0)=1\}.$$
\end{definition}

\begin{definition}
Let $\Omega=\Omega_1\times\Omega_2\subseteq\mathbb{BC}$ be an open set. Then we say that a function $R: \Omega \to \mathbb{BC}$  is $\mathbb{BC}$-rational if  $'R'$ is the quotient of two continuous $\mathbb{BC}$-functions i.e.,
 $$R(Z) = \frac{G(Z)}{H(Z)}~~ \mbox{such~that}~ H(Z) \notin \mathcal{O}_0.$$ 
Also  bicomplex holomorphic rational functions  are product-type, i.e., there exist $R_i: \Omega_i \to \mathbb{C}(\mathbf{i})$ for $i = 1,2$ such that $R(\beta_1 \mathbf{e_1} + \beta_2 \mathbf{e_2}) = R_1(\beta_1)\mathbf{e_1} + R_2(\beta_2)\mathbf{e_2}$ for all $\beta_1 \mathbf{e_1} + \beta_2 \mathbf{e_2} \in \Omega$.
\end{definition}

A function $F(\beta_1 \mathbf{e_1} + \beta_2 \mathbf{e_2}) = F_1(\beta_1 \mathbf{e_1} + \beta_2 \mathbf{e_2})\mathbf{e_1} + F_2(\beta_1 \mathbf{e_1} + \beta_2 \mathbf{e_2})\mathbf{e_2}$  is $\mathbb{BC}$-holomorphic if and only if $F_1{(\beta_1 \mathbf{e_1} + \beta_2 \mathbf{e_2})}$ and $F_{2}{(\beta_1 \mathbf{e_1} + \beta_2 \mathbf{e_2})}$  are  holomorphic functions with respect to only $\beta_1$ and $\beta_2$ respectively  and $F(Z_1 + \mathbf{i} Z_2 ) = G_1Z_1 + \mathbf{i} Z_2 ) + \mathbf{j} G_2(Z_1 + \mathbf{i} Z_2 )$ is $\mathbb{BC}$-holomorphic iff $G_1$ and $G_2$ are holomorphic functions with respect $Z_1$ and $Z_2$ respectively. For the above discussion  we refer to \cite{Alpa},\cite{Luna6} and \cite{Reye}.  Also the concepts of $\mathbb{BC}$-rectifiable curve, $\mathbb{BC}$-Jorden curve, $\mathbb{BC}$-closed curves, $\mathbb{BC}$-Integral are well established in \cite{Luna6} and \cite{Reye}. For an excurtion in bicomplex analysis and its applications, one can refer to  \cite{Alpa}, \cite{Cato3}, \cite{Luna3}, \cite{Luna6}, \cite{Luna8}, \cite{Luna10}, \cite{Pric}, \cite{Reye}, \cite{Sain} and references therein.

\section{Bicomplex Koebe Quarter theorem}

In this section, we study Area Theorem, Bieberbach's Theorem and Koebe Quarter Theorem for bicomplex scalars. Let us denote the class of $\mathbb{BC}$-holomorphic, $\mathbb{BC}$-univalent functions on $\Delta=\{Z:|Z|_{\mathbf{k}}\succ 1 \}$ by $\Sigma$. That is,\\
$\Sigma=\{G:\Delta\rightarrow\mathbb{BC}:G(Z)=Z+B_{0}+B_{1}Z^{-1}+B_{2}Z^{-2}+..., ~\mbox{where}~Z\notin\mathcal{O}_0,~ B_{n}\in\mathbb{BC}~ \mbox{and}~  G(Z) ~ \mbox{is} ~ \mathbb{BC}-\mbox{holomorphic},~ \mathbb{BC}-\mbox{univalent}.\}$

Let $E=\mathbb{BC}\backslash G(\Delta)$, where $G\in \Sigma$. Then, $G$ maps $\Delta$ onto the complement of the $\mathbb{BC}$-compact connected set $E$.\\ 

Before we establish bicomplex version of Area Theorem. We start with the statement of bicomplex version of Green Theoerm \cite[Theorem 5.4, Page-40]{Stef}.

\begin{theorem} \label{TH1}

Let $\Phi_{1},\Phi_{2}:G\subseteq\mathbb{BC}\longrightarrow\mathbb{BC}$ be a $\mathbb{BC}$-holomorphic function of $U$ and $V$ in a product-type domain $G$. Also let $S(\Gamma)$ be a product-type orentable surface in $G$ bounded by the $\mathbb{BC}$-closed curve $\Gamma=\gamma_{1}\mathbf{e_1}+\gamma_{2}\mathbf{e_2}$. Then
\begin{equation}\label{EQ11}
\int_{\Gamma}\Phi_{1}dU+\Phi_{2}dV=\int_{S(\Gamma)}\left(\frac{\partial\Phi_{2}}{\partial U}-\frac{\partial\Phi_{1}}{\partial V}\right)dU\wedge dV.
\end{equation}

\end{theorem}

\begin{theorem}\label{Thrm1}
Suppose $G(Z)=Z+B_{0}+\sum_{n=1}^{\infty}{B_{n}Z^{-n}}\in\Sigma$ such that $Z\notin\mathcal{O}_0$. Then $\sum_{n=1}^{\infty}{n|B_{n}|^{2}_{\mathbf{k}}}\preceq 1$.
\end{theorem}

\begin{proof}
Let $E=\mathbb{BC}\backslash G(\Delta)$, and let $\Gamma_{r}$ be the image under $F~(\mbox{where}~ F\in\mathcal{F})$ of the circle $|Z|_{\mathbf{k}}=r$. In terms of the coefficients of $\mathbb{BC}$-Taylor's series, we want to calculate the area of $E$. Let's estimate $E$ from the outside using the domain $E_{r}=\mathbb{BC}\backslash\{G(Z);|Z|_{\mathbf{k}}\succ r\}$, as $E$ might be extremely erratic. As we know, $G$ is $\mathbb{BC}$-univalent, $\Gamma_{r}$ is smooth simple $\mathbb{BC}$-closed curve which enclose the domain $E_{r}$, so we can write $G(Z)=G_1(\beta_1)\mathbf{e}_{1}+G_2(\beta_2)\mathbf{e}_{2}$, $\Gamma_{r}=\Gamma_{r_1}\mathbf{e}_{1}+\Gamma_{r_2}\mathbf{e}_{2}$ and
$$Area(E_{r})=\underset{E_{r}}{\int} {dU \cdot dV}.$$
Now using theorem \ref{TH1}, taking $\Phi_1=V$ and $\Phi_2=0$ and using polar representation of bicomplex numbers \cite[Page-65]{Luna6}, we have

\begin{eqnarray*}
Area(E_{r})&=&\int_{\Gamma_{r}}V\cdot dU\\
&=&\int_{|Z|_{\mathbf{k}}=r}{G(Z)G'(Z)dZ\wedge dZ^{\dag}}\\
&=& \int_{0}^{2\pi}{r e^{\mathbf{i}\theta } G^{'}(r e^{\mathbf{i}\theta})G(r e^{\mathbf{i}\theta})d\theta}\\
&=& \int_{0}^{2\pi}{r_1 e^{\mathbf{i}\theta_1} G_1^{'}(r_1 e^{\mathbf{i}\theta_1})G_1(r_1 e^{\mathbf{i}\theta_1})d\theta_1\mathbf{e}_{1}}+\int_{0}^{2\pi}{r_2 e^{\mathbf{i}\theta_2} G_2^{'}(r_2 e^{\mathbf{i}\theta_2})G_2(r_2 e^{\mathbf{i}\theta_2})d\theta_2\mathbf{e}_{1}}.
\end{eqnarray*}

Writing $G$ and $G^{'}$ as their $\mathbb{BC}$-Taylor series expansion \cite[Page-208]{Luna6}, we get

\begin{eqnarray*}
\mbox{Area}(E_r)&=&\int_{0}^{2\pi}{( r_1e^{\mathbf{i}\theta_1}-\sum_{n_1=1}^{\infty}{n_1B_{n_1}r^{-n_1}e^{-\mathbf{i}n_1\theta_1}) ( r_1e^{\mathbf{i}\theta_1}}-\sum_{m_1=0}^{\infty}{m_1B_{m_1}r_1^{-m_1}e^{-\mathbf{i}m_1\theta_1}) d\theta_1}\mathbf{e}_{1}}\\
&&+\int_{0}^{2\pi}{( r_2e^{\mathbf{i}\theta_2}-\sum_{n_2=1}^{\infty}{n_2B_{n_2}r^{-n_2}e^{-\mathbf{i}n_2\theta_2}) ( r_2e^{\mathbf{i}\theta_2}}-\sum_{m_2=0}^{\infty}{m_2B_{m_2}r_2^{-m_2}e^{-\mathbf{i}m_2\theta_2}) d\theta_2}\mathbf{e}_{2}}\\
&=&\sum_{l=1}^{2}\left({\int_{0}^{2\pi}{ ( r_le^{\mathbf{i}\theta_l}-\sum_{n_l=1}^{\infty}{n_lB_{n_l}r_l^{-n_l}e^{-\mathbf{i}n_l\theta_l})(r_le^{\mathbf{i}\theta_l}}-\sum_{m_l=0}^{\infty}{m_lB_{m_l}r_l^{-m_l}e^{-\mathbf{i}m_l\theta_l}) d\theta_l\mathbf{e}_{l}}}}\right). 
\end{eqnarray*}

Using Area theorem \cite[Theorem 14.13, Page-286]{Rudi2}, we have,
$$\sum_{n=1}^{m_l}{n_l|B_{n}|^{2}r_l^{2n_l}< r_1^{2}\mathbf{e}_{l}},~~\mbox{where}~~l=1,2.$$
Finally, letting $r_1,r_1\rightarrow 1$ and $m_1,m_2\rightarrow 0$, we obtain
$$\sum_{n=1}^{\infty}n_1|B_{n_1}|^{2}\mathbf{e}_{1}+\sum_{n=1}^{\infty}n_2|B_{n_2}|^{2}\mathbf{e}_{2}\preceq 1.$$
Thus,
$$\sum_{n=1}^{\infty}n|B_n|_{\mathbf{k}}^{2}\preceq 1.$$
\end{proof}

\begin{corollary}\label{Coro1}
If $G\in\Sigma$, then $|B_1|_{\mathbf{k}}\preceq 1$ with equality if and only if $G$ has the form
$$G(Z)=Z+B_{0}+\frac{B_{1}}{Z},~~ |B_1|_{\mathbf{k}} = 1,~where~Z\notin\mathcal{O}_0.$$
\end{corollary}

\begin{proof}
If $G\in\Sigma$, by Theorem \ref{Thrm1}, we have
$$\sum_{n=1}^{\infty}{n|B_{n}|_{\mathbf{k}}^{2}\preceq 1}.$$
If follows that, for every $n \geq 1,~ n|B_{n}|_{\mathbf{k}}^{2}\preceq 1.$ In particular, for $n=1,~ |B_{1}|_{\mathbf{k}}^{2}\preceq 1$ and so $|B_{1}|_{\mathbf{k}}\preceq 1.$\\
On the other hand, if $|B_1|_{\mathbf{k}} = 1,$ then necessarily $B_n = 0$ for all $n > 1.$ Therefore,
$G$ has the form
$$G(Z) = Z + B_0 + B_1Z^{-1}, ~\mbox{with}~ |B_1|_{\mathbf{k}} = 1.$$
Converse part directly follows from definition of $G$.
\end{proof}

\begin{definition}\cite[Definition 15, Page-5]{ghos}
Let $A, B, C, D \in \mathbb{BC}$ and A, B, C, D, $AD-BC \notin\mathcal{O}_{0}$. Then the mapping $F : \overline{\mathbb{BC}}\rightarrow\overline{\mathbb{BC}}$ defined by
\begin{equation}\label{EQ11}
F (W) = \frac{AZ + B}{CZ + D},~where~CZ+D\notin\mathcal{O}_0
\end{equation}
with the additional agreement that $F (\infty) = \frac{A}{C}$ and $F\left(\frac{-D}{C}\right)= \infty$ is called a bicomplex Mobius transformation or bicomplex Mobius map.
\end{definition}

\begin{remark}
Denoting $Z =\beta_1 \mathbf{e}_{1}+\beta_2\mathbf{e}_{2},~ A=A_{1}\mathbf{e}_{1}+A_{2}\mathbf{e}_{2},~ B=B_{1}\mathbf{e}_{1}+B_{2}\mathbf{e}_{2},~ C=C_{1}\mathbf{e}_{1}+C_{2}\mathbf{e}_{2},~ D=D_{1}\mathbf{e}_{1}+D_{2}\mathbf{e}_{2}$ and rewriting (\ref{EQ11}) in the idempotent form we get:
\begin{eqnarray*}
F (\beta_1 \mathbf{e}_{1}+\beta_2\mathbf{e}_{2}) &=& \left(\frac{A_{1}\beta_{1} + B_{1}}{C_{1}\beta_{1} + D_{1}}\right)\mathbf{e}_{1}+\left(\frac{A_{2}\beta_{2} + B_{2}}{C_{2}\beta_{2} + D_{2}}\right)\mathbf{e}_{2}\\
&=&\left(\frac{A_{1}}{C_{1}} + \frac{B_{1}C_{1}-A_{1}D_{1}}{C^{2}_{1}\left(\beta_{1}+\frac{D_{1}}{C_{1}}\right)}\right)\mathbf{e}_{1}+\left(\frac{A_{2}}{C_{2}} + \frac{B_{2}C_{2}-A_{2}D_{2}}{C^{2}_{2}\left(\beta_{2}+\frac{D_{2}}{C_{2}}\right)}\right)\mathbf{e}_{2}.
\end{eqnarray*}
\begin{itemize}
\item[$\bullet$] If $C_{1}=0$ and $C_{2}=0$, then
$F (\beta_1 \mathbf{e}_{1}+\beta_2\mathbf{e}_{2})$ is a $\mathbb{BC}-$linear function.
\item[$\bullet$] If $C_{1}\neq 0$ and $C_{2}=0$, then\\
$F (\beta_1 \mathbf{e}_{1}+\beta_2\mathbf{e}_{2})=\left(\frac{A_{1}}{C_{1}} + \frac{B_{1}C_{1}-A_{1}D_{1}}{C^{2}_{1}\left(\beta_{1}+\frac{D_{1}}{C_{1}}\right)}\right)\mathbf{e}_{1}+(\mathbb{C}(\mathbf{i})-$linear$)\mathbf{e}_{2}$.\\
$F\left(-\frac{D_{1}}{C_{1}}\mathbf{e}_{1}\right)=\infty\mathbf{e}_{1}$\\
also $F\left(\infty\mathbf{e}_{1}\right)=\frac{A_{1}}{C_{1}}\mathbf{e}_{1}$.
\item[$\bullet$] If $C_{1}= 0$ and $C_{2}\neq0$, then\\
$F (\beta_1 \mathbf{e}_{1}+\beta_2\mathbf{e}_{2})=(\mathbb{C}(\mathbf{i})-$linear$)\mathbf{e}_{1}$  $+\left(\frac{A_{2}}{C_{2}} + \frac{B_{2}C_{2}-A_{2}D_{2}}{C^{2}_{2}\left(\beta_{2}+\frac{D_{2}}{C_{2}}\right)}\right)\mathbf{e}_{2}$.\\
$F\left(-\frac{D_{2}}{C_{2}}\mathbf{e}_{2}\right)=\infty\mathbf{e}_{2}$\\
also $F\left(\infty\mathbf{e}_{2}\right)=\frac{A_{2}}{C_{2}}\mathbf{e}_{2}$.
\item[$\bullet$] If $C_{1}\neq 0$ and $C_{2}\neq 0$, then\\
$F (\beta_1 \mathbf{e}_{1}+\beta_2\mathbf{e}_{2})=\left(\frac{A_{1}}{C_{1}} + \frac{B_{1}C_{1}-A_{1}D_{1}}{C^{2}_{1}\left(\beta_{1}+\frac{D_{1}}{C_{1}}\right)}\right)\mathbf{e}_{1}+\left(\frac{A_{2}}{C_{2}} + \frac{B_{2}C_{2}-A_{2}D_{2}}{C^{2}_{2}\left(\beta_{2}+\frac{D_{2}}{C_{2}}\right)}\right)\mathbf{e}_{2}.$\\
$F\left(-\frac{D_{1}}{C_{1}}\mathbf{e}_{1}-\frac{D_{2}}{C_{2}}\mathbf{e}_{2}\right)=\infty\mathbf{e}_{1}+\infty\mathbf{e}_{2}$\\
also $F\left(\infty\mathbf{e}_{1}+\infty\mathbf{e}_{2}\right)=\frac{A_{1}}{C_{1}}\mathbf{e}_{1}+\frac{A_{2}}{C_{2}}\mathbf{e}_{2}$.
\end{itemize}
\end{remark}

Now we prove bicomplex Bieberbach’s Theorem.
\begin{theorem}\label{Thrm2}
If $F \in \mathcal{F}$, then $|A_2|_{\mathbf{k}} \preceq 2$, with equality if and
only if $F$ is a rotation of the $\mathbb{BC}$-Koebe function.
\end{theorem}

\begin{proof}
Given $F(Z) = Z + \sum_{n=2}^{\infty}{A_{n}Z^{n}}$, we construct the following auxiliary functions
$$G(Z) = \sqrt{F\left(Z^{2}\right)} ~\mbox{and}~ H(Z) = \frac{1}{G\left(\frac{1}{Z}\right)}~\mbox{such~that}~Z,G\left(\frac{1}{Z}\right)\notin\mathcal{O}_{0}.$$
Since, $G(Z)$ is a square root transformation of $F$ \cite[Theorem 3.7, Page-14]{Arif2}, so $G \in \mathcal{F}$. Now write it's coefficients in terms of the coefficients of $F$.
Let $G\left(Z\right) = Z + B_{3}Z^{3} + B_{5}Z^{5} + \cdot\cdot\cdot .$ We know that $G\left(Z\right)^{2} = F\left(Z^{2}\right)$. Grouping the coefficients by the power of $Z$, we obtain:\\
$n = 4 : A_{2} = 2B_{3} \Rightarrow B_{3} =\frac{A_2}{2}$\\
$n = 6 : A_{3} = 2B_{5} + B^{2}_{3} \Rightarrow B_{5} =\frac{1}{2}\left(A_3 - \frac{A^2_2}{4}\right)$\\
$\cdot\cdot\cdot.$\\
Thus, we re-write $G(Z)$ as\\
$$G(Z) = Z +\frac{A_3}{2}Z^3 +\frac{1}{2}\left(A_3 -\frac{A_2^2}{4}\right)Z^5 + \cdot\cdot\cdot.$$
The function $H(Z)$ is $\mathbb{BC}$-univalent, and also study its $\mathbb{BC}$-Laurent series expansion.\\
We compute: $G\left(\frac{1}{Z}\right)=\frac{1}{Z}+\frac{B_3}{Z^3} +\frac{B_5}{Z^5} + \cdot\cdot\cdot.$ Observe that, as $G \in \mathcal{F}, G\left(\frac{1}{Z}\right)$ is $\mathbb{BC}$-holomorphic and $\mathbb{BC}$-univalent for $|Z|_{\mathbf{k}} \succ 1.$ Consequently, $H(Z) = \frac{1} {G\left(\frac{1}{Z}\right)}$ is $\mathbb{BC}$-univalent and
$\mathbb{BC}$-holomorphic in $\Delta$ and, moreover, $H(\infty) = \infty.$ Therefore, $H \in \Sigma$, and has the form
$$H(Z) = Z + C_0 +\frac{C_1}{Z}+\frac{C_2}{Z^2} + \cdot\cdot\cdot.$$
Let's compute the coefficients of $H(Z)$ in terms of the $A_n$.\\
Since $H(Z) = \frac{1}{G\left(\frac{1}{Z}\right)}$, we have $G\left(\frac{1}{Z}\right)H(Z) = 1$. Writing $\mathbb{BC}$-Taylor's development of the functions, we get
$$1 =\left(\frac{1}{Z}+\frac{B_3}{Z^3} +\frac{B_5}{Z^5} + \cdot\cdot\cdot\right)\left( Z + C_0 +\frac{C_1}{Z}+\frac{C_2}{Z^2} + \cdot\cdot\cdot \right).$$
Now, multiplying and grouping the coefficients by the power of $Z$, and equating to zero, we get:\\
$n = -1 : C_0 = 0$\\
$n = -2 : C_1 + B_3 = 0 \Rightarrow C_1 = -B_3 = -\frac{A_2}{2}$\\
$n = -3 : C_2 + B_3C_0 = 0 \Rightarrow C_2 = 0$\\
$\cdot\cdot\cdot$\\
Finally, we obtain
$$H(Z) = Z -\frac{A_2}{2Z}+ \cdot\cdot\cdot.$$
As $H \in \Sigma,$ by Corollary \ref{Coro1}, we have that $\left|\frac{A_2}{2}\right|_{\mathbf{k}} \preceq 1$, that is, $\left|A_2\right|_{\mathbf{k}} \preceq 2$, with equality holds if and only if 
$$H(Z) = Z +\frac{B}{Z}, \left|B\right|_{\mathbf{k}} = 1.$$
It now remains to be seen that this is equivalent to $F$, being a rotation of the $\mathbb{BC}$-Koebe function.\\
Now \\
$H(Z) = \frac{1}{G\left(\frac{1}{Z}\right)}$ is equivalent to $G(Z) = \frac{1}{H\left(\frac{1}{Z}\right)}, Z \in \mathbb{D}_{\mathbb{BC}}.$\\
So, we have\\
$H(Z) = Z +\frac{B}{Z}$ if and only if $H\left(\frac{1}{Z}\right) = \frac{1}{Z}+ BZ.$\\
Thus,\\
$G(Z) = \frac{1}{\frac{1}{Z} + BZ},$ that is, $G(Z) = \frac{Z}{1 + BZ^2}.$\\
As $G(Z) = \sqrt{ F(Z^2)}$, we have 
$$G(Z)^2 =\frac{Z^2}{(1 + BZ^2)^2} = F(Z^2).$$ 
With the change of variable $t = Z^2,$ we obtain
$$F(t) = \frac{t}{(1 + Bt)^2}, \left|B\right|_{\mathbf{k}} = 1, t \in \mathbb{D}_{\mathbb{BC}}$$
which is a rotation of the $\mathbb{BC}$-Koebe function.
\end{proof}

Now we establish bicomplex version of Koebe Quarter Theorem.
\begin{theorem}
The range of every $\mathbb{BC}$-univalent $\mathbb{BC}$-holomorphic function $F:\mathbb{D}_{\mathbb{BC}}\rightarrow\mathbb{BC}$ from bicomplex unit disk $\mathbb{D}_{\mathbb{BC}}$ onto a subset of the bicomplex plane contains the bicomplex disk $\mathbb{D}(0, \frac{1}{4})$.
\end{theorem}

\begin{proof} 
Let $W \in \mathbb{BC}\backslash\mathcal{O}_{0}$ be a point. Let $F \in \mathcal{F}$ be a function that omits the value $W$, and consider the function
$$G(Z) = \frac{W F(Z)}{W - F(Z)},~\mbox{where}~W-F(Z)\notin\mathcal{O}_0.$$
Firstly, we check that $G$ is $\mathbb{BC}$-univalent and $\mathbb{BC}$-holomorphic in $\mathbb{D}_{\mathbb{BC}}$. We can write
$$G(Z) = (H \circ F)(Z),~ \mbox{where} ~H(Z) = \frac{WZ}{W - Z}~\mbox{and~also}~W-Z\notin\mathcal{O}_0.$$
Since $F$ is $\mathbb{BC}$-univalent and $\mathbb{BC}$-holomorphic in $\mathbb{D}_{\mathbb{BC}}$, it only needs to be check that $H(Z)$ is $\mathbb{BC}$-univalence and $\mathbb{BC}$-holomorphic. Suppose that $H(Z_1) = H(Z_2).$ So,
$$\frac{WZ_1}{W - Z_1}=\frac{WZ_2}{W - Z_2}.$$
Then,
$$WZ_1 - Z_1Z_2 = Z_2W - Z_1Z_2 ~\mbox{and ~hence},~ Z_1 = Z_2.$$
Clearly, $H$ is $\mathbb{BC}$-holomorphic for every $Z \neq W.$ Then, by \cite[Theorem 2, Page-5]{ghos}, $G$ is $\mathbb{BC}$-holomorphic and $\mathbb{BC}$-univalent in $\mathbb{D}_{\mathbb{BC}}$. Now, we have $G(Z)\in\mathcal{F}$, because if\\
$G(Z) = \frac{W F(Z)}{W - F(Z)}~\mbox{then}~ G(0) = 0$,\\
$G'(Z) = \frac{W^2F'(Z)}{(W - F(Z))^2} ~\mbox{then}~ W'(0) = 1$ and\\
$G''(Z) = \frac{W^2(F''(Z)(W - F(Z)) + 2 F'(Z)^2)}{(W - F(Z))^3}$ implies $G''(0) = F''(0) + \frac{2}{W}$, where $W\notin\mathcal{O}_{0}$. Thus
$$G(Z) = Z +\left(A_2 +\frac{1}{W}\right)Z^2 + \cdot\cdot\cdot.$$
Since, $G(Z)\in\mathcal{F}$, by Theorem \ref{Thrm2} we have,
$$\left|A_2 +\frac{1}{W}\right|_{\mathbf{k}} \preceq 2.$$
Now, using the triangle inequality we get
$$\left|\frac{1}{W}\right|_{\mathbf{k}} \preceq \left|\frac{1}{W}+ A_2\right|_{\mathbf{k}}+ |A_2|_{\mathbf{k}} \preceq 2 + 2 = 4, \mbox{~and ~therefore~} |W|_{\mathbf{k}} \succeq \frac{1}{4}.$$
Now in view of Theorem \ref{Thrm2}, $|W|_{\mathbf{k}} = \frac{1}{4}$ holds if and only if $G$ is a rotation of $\mathbb{BC}$-Koebe function and by the definition of $G$ this holds if and only if $F$ is a rotation of the $\mathbb{BC}$-Koebe function.
\end{proof}

\begin{remark}
	Since holomorphic functions have idempotent representation, it is easy to  prove that the Bieberbach's conjecture  for bicomplex scalers is true if and only if it is true for complex scalers.  
\end{remark}

\section{Bicomplex Mergelyan's Theorem}
In this section we prove the bicomplex version of  Mergelyan's approximation theorem. First of all we fix some notations.  Let us denote the space of all  bicomplex valued functions which are  $\mathbb{BC}$-holomorphic  in the  interior of  product type compact set $K = K_1 \times K_2  $ and are  $\mathbb{BC}$-continuous on   $K = K_1 \times K_2 $,  where $K_1 ~ and ~ K_2 $  are  domains in $\mathbb{C}(\mathbf{i})$ as,   
$$ \mathcal{A}(K) = \mathcal{A}({K}_1) \times \mathcal{A}({K}_2) = \mathcal{A}({K}_1) \mathbf{e_1} + \mathcal{A}({K}_2)\mathbf{e_2},  $$
where $\mathcal{A}({K}_1)$ and  $\mathcal{A}({K}_2)$ are space of complex valued functions which are  holomorphic in   the interior of ${K}_1 ~ and ~ {K}_2 $ and continuous on  ${K}_1 ~ and ~ {K}_2 $ respectively.
Also the product-type compact subset $K \subset \mathbb{BC}$ can be written as,
 $$K= K_1 \times K_2 = {K}_1 \mathbf{e_1} + {K}_2 \mathbf{e_2},~~\mbox{where}~~ K_1, K_2 \subset \mathbb{C}(\mathbf{i}).$$
 If ${K}_1$ is finitely connected and  ${K}_2$ is simply connected then, $K= K_1 \times K_2$ is finitely connected as shown in the Fig. \ref{fig:1}.
 \begin{figure}[h]
 	\centering
 	\includegraphics[width=12cm, height=5cm]{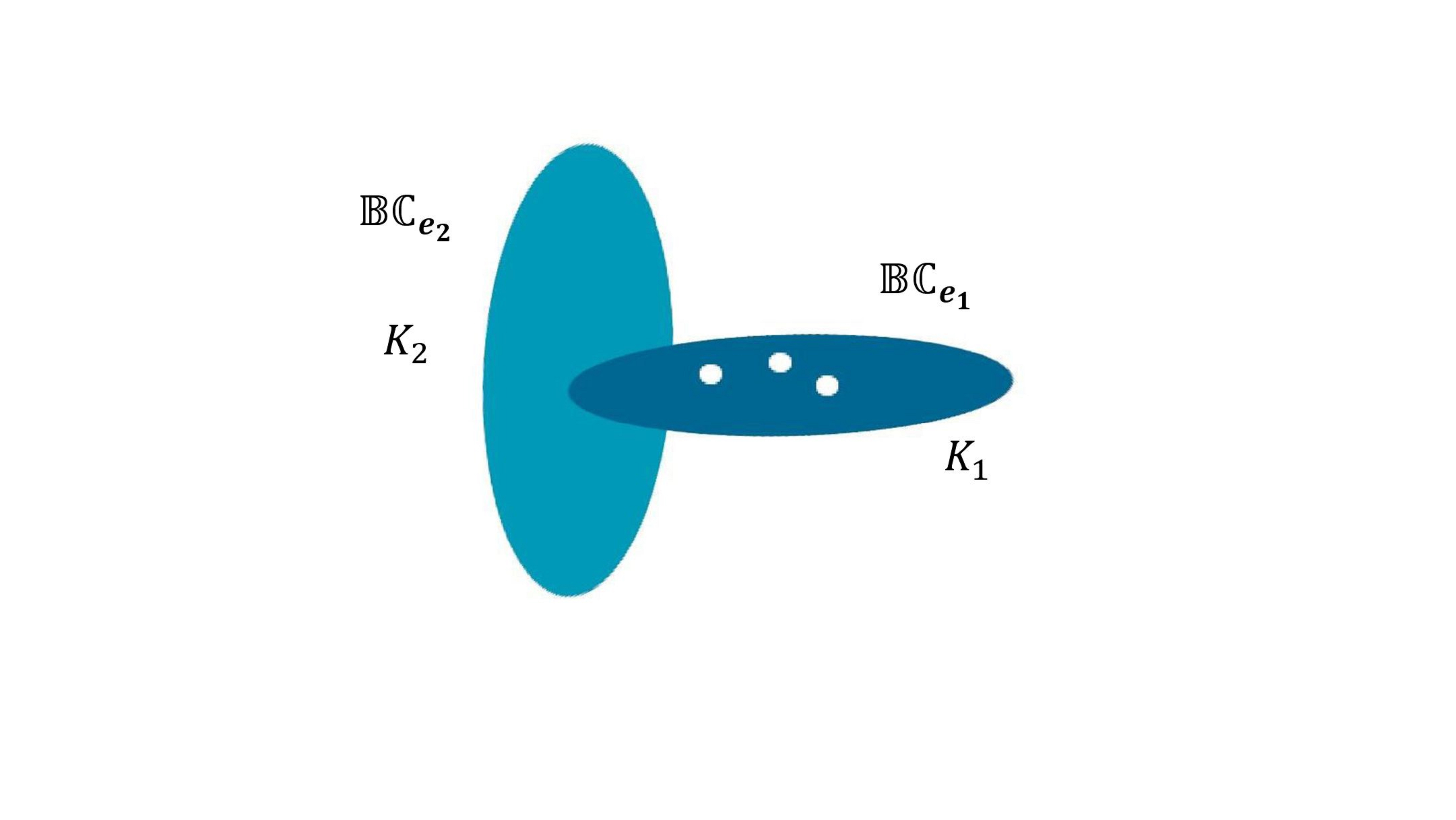}
 	\caption{A section of $K= K_1 \times K_2$ }
 	\label{fig:1}
 \end{figure}

If $K_1$ is compact subset of $\mathbb{C}(\mathbf{i})$, then $\mathbb{C}(\mathbf{i})\setminus K_1$ can be connected or have finitely many components or  infinitely many components. For a compact set $K= K_1 \times K_2 \subset \mathbb{BC}$, we have,

 $$  \mathbb{BC} \setminus K = (\mathbb{C}(\mathbf{i})\setminus K_1) \times (\mathbb{C}(\mathbf{i})\setminus K_2). $$
 Then we have following four possibilities for  $  \mathbb{BC} \setminus K $:
 \begin{itemize}
 	 \item[(i)] If both  ($\mathbb{C}(\mathbf{i})\setminus K_1), (\mathbb{C}(\mathbf{i})\setminus K_2)$ have finitely many  or infinitely many components, then $  \mathbb{BC} \setminus K $ have finitely many  or infinitely many components respectively. In that case let  
	$$ \mathcal{T}_1 =   \mathbb{BC} \setminus K = (\mathbb{C}(\mathbf{i})\setminus K_1) \times (\mathbb{C}(\mathbf{i})\setminus K_2) $$ 

  \item[(ii)] If  $\mathbb{C}(\mathbf{i})\setminus K_1$ have finitely many components or infinitely many components and  $ \mathbb{C}(\mathbf{i})\setminus K_2$ is  connected, then $  \mathbb{BC} \setminus K $ also have finitely many  or infinitely many components respectively and we denote it by $\mathcal{T}_2$  
	$$ \mathcal{T}_2 =   \mathbb{BC} \setminus K = (\mathbb{C}(\mathbf{i})\setminus K_1) \times (\mathbb{C}(\mathbf{i})\setminus K_2). $$ 
   \item[(iii)] If  $\mathbb{C}(\mathbf{i})\setminus K_1$  is  connected and  $ \mathbb{C}(\mathbf{i})\setminus K_2$ have finitely many or infinitely many components, then $  \mathbb{BC} \setminus K $ also have finitely many  or infinitely many components respectively and we denote it by $\mathcal{T}_3$
     $$ \mathcal{T}_3 =   \mathbb{BC} \setminus K = (\mathbb{C}(\mathbf{i})\setminus K_1) \times (\mathbb{C}(\mathbf{i})\setminus K_2).$$ 
   	\item[(iv)] If both  $(\mathbb{C}(\mathbf{i})\setminus K_1), (\mathbb{C}(\mathbf{i})\setminus K_2)$ are connected, then $  \mathbb{BC} \setminus K $ is also connected. In that case let  
		$$ \mathcal{T}_4 =   \mathbb{BC} \setminus K = (\mathbb{C}(\mathbf{i})\setminus K_1) \times (\mathbb{C}(\mathbf{i})\setminus K_2).$$ 
 \end{itemize}
Now we prove the bicomplex Mergelyan's Thoreom. For complex version of Mergelyan's Theorem we refer to \cite{Green} and \cite{Rudi2}.

\begin{theorem}\label{Mergelyan's} Let K be a product-type compact subset of $\mathbb{BC}$ and    $F(Z) \in \mathcal{A}(K)$.  
\begin{itemize}	
\item[(i)] If $  \mathbb{BC} \setminus K $  is of the type $\mathcal{T}_1$ having finitely many components and $P \subset \mathcal{T}_1$ contains one point from each connected component of $\mathcal{T}_1$, then $F(Z)$ can be approximated  on $K$  by a bicomplex rational function $R(Z)$ with bicomplex poles only in $P$ i.e.,
$$\sup_{z\in K} |F(Z) - R(Z)|_{\mathbf{k}} \prec \epsilon .$$
\item[(ii)] If $  \mathbb{BC} \setminus K $  is of the type $\mathcal{T}_2$ having finitely many components, then  $F(Z)$ can be approximated  on $K$  by a bicomplex rational function $R(Z) = R_1(Z_1)\mathbf{e_1} + P_2(Z_2)\mathbf{e_2} $, where $R_1(Z_1)$ is rational function in $\mathbb{C}(\mathbf{i})$ and $P_2(Z_2)$ is polynomial in $\mathbb{C}(\mathbf{i})$, with bicomplex  poles   $ \beta_1 \mathbf{e_1} + \infty\mathbf{e_2} \in \overline{\mathbb{BC}}\setminus( K \cup \mathcal{O}) $  and $\beta_1 \in (\mathbb{C}(\mathbf{i})\setminus K_1)$ such that,
$$\sup_{z\in K} |F(Z) - R(Z)|_{\mathbf{k}}\prec \epsilon .$$
\item[(iii)] If $  \mathbb{BC} \setminus K $  is of the type $\mathcal{T}_3$ having finitely many components, then  $F(Z)$ can be approximated  on $K$  by a bicomplex rational function $R(Z) = P_1(Z_1)\mathbf{e_1} + R_2(Z_2)\mathbf{e_2} $ where $P_1(Z_1)$ is polynomial in $\mathbb{C}(\mathbf{i})$ and $R_2(Z_2)$ is rational function in $\mathbb{C}(\mathbf{i})$, with bicomplex poles $ \infty\mathbf{e_1} + \beta_2\mathbf{e_2} \in \overline{\mathbb{BC}}\setminus( K \cup \mathcal{O})$ and $\beta_2 \in (\mathbb{C}(\mathbf{i})\setminus K_2)$ such that,
$$\sup_{z\in K} |F(Z) - R(Z)|_{\mathbf{k}} \prec \epsilon .$$

\item[(iv)] If $  \mathbb{BC} \setminus K $  is of the type $\mathcal{T}_4$, then $F(Z)$ can be approximated  on $K$  by a bicomplex polynomial $P(Z)$ with bicomplex pole $ (\infty\mathbf{e_1} + \infty\mathbf{e_2}) \in \overline{\mathbb{BC}}\setminus( K \cup \mathcal{O})$ such that,
$$\sup_{z\in K} |F(Z) - P(Z)|_{\mathbf{k}} \prec \epsilon .$$		
\end{itemize}
\end{theorem}
\begin{proof}
	\begin{itemize}
		
\item[(i)] We have $F(Z) = F_1(Z_1)\mathbf{e_1} + F_2(Z_2)\mathbf{e_2} \in   \mathcal{A}(K) $ and   $  \mathbb{BC} \setminus K $  is of the type $\mathcal{T}_1$. So $F_j(Z_j)  \in   \mathcal{A}(K_j)$ for $ j = 1, 2 $ and  $\mathbb{C}(\mathbf{i})\setminus K_j $ for $ j = 1, 2 $ have finitely  many components. Then  by Mergelyan's approximation theorem for complex  function,  there exists a rational function $R_j(Z_j) $ for $ j = 1, 2 $ with poles in $P_j \in \mathbb{C}(\mathbf{i})\setminus K_j $  for $ j = 1, 2 $ such that,
	$$ \sup_{Z_j\in K_j} |F_j(Z_j) - R_j(Z_j)| < \epsilon~~for~~  j = 1, 2  .$$
Now,
\begin{eqnarray*}
	 \sup_{z\in K} |F(Z) - R(Z)|_{\mathbf{k}} &\preceq&  \sup_{Z_1\in K_1}|F_1(Z_1) - R_1(Z_1)| \mathbf{e_1} + \sup_{Z_2\in K_2}|F_2(Z_2) - R_2(Z_2) | \mathbf{e_2}\\ &\prec& \epsilon \mathbf{e_1}	+  \epsilon\mathbf{e_2} \\
	 &\prec&  \epsilon,
 \end{eqnarray*}
where $R(Z) = R_1(Z_1)\mathbf{e_1} + R_2(Z_2)\mathbf{e_2}$ is a bicomplex  rational function with bicomplex poles $ P = P_1 \mathbf{e_1}  + P_2\mathbf{e_2}  \in \mathbb{BC}\setminus( K \cup \mathcal{O}) $.

\item[(ii)]  We have $F(Z) = F_1(Z_1)\mathbf{e_1} + F_2(Z_2)\mathbf{e_2} \in   \mathcal{A}(K) $ and   $  \mathbb{BC} \setminus K $  is of the type $\mathcal{T}_2$. So $F_1(Z_1)  \in   \mathcal{A}(K_1)$  and  $\mathbb{C}(\mathbf{i})\setminus K_1 $  have finitely  many components. Then  by Mergelyan's approximation theorem for complex  function,  there exist a rational function $R_1(Z_1) $  with poles in $P_1 \in (\mathbb{C}(\mathbf{i})\setminus K_1) $  such that,
		$$ \sup_{Z_1\in K_1} |F_1(Z_1) - R_1(Z_1)| < \epsilon .$$
Also, $F_2(Z_2)  \in   \mathcal{A}(K_2)$  and  $(\mathbb{C}(\mathbf{i})\setminus K_2) $  is connected, then again by Mergelyan's approximation theorem for complex  function,  there exist a polynomial $P_2(Z_2) $  with only pole in $\infty \in (\mathbb{C}(\mathbf{i})\setminus K_2) $  such that,
		$$ \sup_{Z_2\in K_2} |F_2(Z_2) - P_2(Z_2)| < \epsilon .$$
Now,
\begin{eqnarray*}
			\sup_{z\in K} |F(Z) - R(Z)|_{\mathbf{k}} &\preceq&  \sup_{Z_1\in K_1}|F_1(Z_1) - R_1(Z_1)| \mathbf{e_1} + \sup_{Z_2\in K_2}|F_2(Z_2) - P_2(Z_2) | \mathbf{e_2}\\ &\prec& \epsilon \mathbf{e_1}	+  \epsilon\mathbf{e_2} \\
			&\prec&  \epsilon,
\end{eqnarray*}
where $R(Z) = R_1(Z_1)\mathbf{e_1} + P_2(Z_2)\mathbf{e_2}$ is a bicomplex  rational function with bicomplex poles as bicomplex infinity $ \beta_1 \mathbf{e_1} + \infty\mathbf{e_2} \in \overline{\mathbb{BC}}\setminus( K \cup \mathcal{O}) $,  with $\beta_1 \in (\mathbb{C}(\mathbf{i})\setminus K_1)$.

\item[(iii)]  The proof of this part is  similar to that of case (ii).

\item[(iv)] We have $F(Z) = F_1(Z_1)\mathbf{e_1} + F_2(Z_2)\mathbf{e_2} \in   \mathcal{A}(K) $ and   $  \mathbb{BC} \setminus K $  is of the type $\mathcal{T}_4$. So $F_j(Z_j)  \in   \mathcal{A}(K_j)$ for $ j = 1, 2 $ and  $\mathbb{C}(\mathbf{i})\setminus K_j $ for $ j = 1, 2 $ is connected. Then  by Mergelyan's approximation theorem for complex  function,  there exist a polynomials $P_j(Z_j) $ for $ j = 1, 2 $ with poles as  $\infty \in \mathbb{C}(\mathbf{i})\setminus K_j $  for $ j = 1, 2 $ such that,
$$ \sup_{Z_j\in K_j} |F_j(Z_j) - R_j(Z_j)| < \epsilon~~for~~  j = 1, 2  .$$
Now,
\begin{eqnarray*}
	\sup_{z\in K} |F(Z) - P(Z)|_{\mathbf{k}} &\preceq&  \sup_{Z_1\in K_1}|F_1(Z_1) - P_1(Z_1)| \mathbf{e_1} + \sup_{Z_2\in K_2}|F_2(Z_2) - P_2(Z_2) | \mathbf{e_2}\\ &\prec& \epsilon \mathbf{e_1}	+  \epsilon\mathbf{e_2} \\
	&\prec&  \epsilon,
\end{eqnarray*}
where $P(Z) = P_1(Z_1)\mathbf{e_1} + P_2(Z_2)\mathbf{e_2}$ is a bicomplex polynomial with bicomplex pole as bicomplex infnity $ \infty \mathbf{e_1} + \infty\mathbf{e_2} \in \overline{\mathbb{BC}}\setminus( K \cup \mathcal{O}) $. This complete proof.
\end{itemize}
\end{proof}

\begin{theorem}\label{Mergelyan's} Let K be a product-type compact subset of $\mathbb{BC}$ and    $F(Z) \in \mathcal{A}(K)$.  
\begin{itemize}	
\item[(i)] If $\mathbb{BC}\setminus( K \cup \mathcal{O})$  is of the type $\mathcal{T}_1$ having finitely many components and $P \subset \mathcal{T}_1$ contains one point from each connected component of $\mathcal{T}_1$, then $F(Z)$ can be approximated  on $K$  by a bicomplex rational function $R(Z)$ with bicomplex poles only in $P$ i.e.,
$$\sup_{z\in K} |F(Z) - R(Z)|_{\mathbf{k}} \prec \epsilon .$$
\item[(ii)] If $\mathbb{BC}\setminus( K \cup \mathcal{O})$  is of the type $\mathcal{T}_2$ having finitely many components, then  $F(Z)$ can be approximated  on $K$  by a bicomplex rational function $R(Z) = R_1(Z_1)\mathbf{e_1} + P_2(Z_2)\mathbf{e_2} $, where $R_1(Z_1)$ is rational function in $\mathbb{C}(\mathbf{i})$ and $P_2(Z_2)$ is polynomial in $\mathbb{C}(\mathbf{i})$, with bicomplex  poles   $ \beta_1 \mathbf{e_1} + \infty\mathbf{e_2} \in \overline{\mathbb{BC}}\setminus( K \cup \mathcal{O}) $  and $\beta_1 \in (\mathbb{C}(\mathbf{i})\setminus K_1)$ such that,
$$\sup_{z\in K} |F(Z) - R(Z)|_{\mathbf{k}}\prec \epsilon .$$
\item[(iii)] If $\mathbb{BC}\setminus( K \cup \mathcal{O})$  is of the type $\mathcal{T}_3$ having finitely many components, then  $F(Z)$ can be approximated  on $K$  by a bicomplex rational function $R(Z) = P_1(Z_1)\mathbf{e_1} + R_2(Z_2)\mathbf{e_2} $ where $P_1(Z_1)$ is polynomial in $\mathbb{C}(\mathbf{i})$ and $R_2(Z_2)$ is rational function in $\mathbb{C}(\mathbf{i})$, with bicomplex poles $ \infty\mathbf{e_1} + \beta_2\mathbf{e_2} \in \overline{\mathbb{BC}}\setminus( K \cup \mathcal{O})$ and $\beta_2 \in (\mathbb{C}(\mathbf{i})\setminus K_2)$ such that,
$$\sup_{z\in K} |F(Z) - R(Z)|_{\mathbf{k}} \prec \epsilon .$$

\item[(iv)] If $\mathbb{BC}\setminus( K \cup \mathcal{O})$  is of the type $\mathcal{T}_4$, then $F(Z)$ can be approximated  on $K$  by a bicomplex polynomial $P(Z)$ with bicomplex pole $ (\infty\mathbf{e_1} + \infty\mathbf{e_2}) \in \overline{\mathbb{BC}}\setminus( K \cup \mathcal{O})$ such that,
$$\sup_{z\in K} |F(Z) - P(Z)|_{\mathbf{k}} \prec \epsilon .$$		
\end{itemize}
\end{theorem}
\begin{proof}
	\begin{itemize}
		
\item[(i)] We have $F(Z) = F_1(Z_1)\mathbf{e_1} + F_2(Z_2)\mathbf{e_2} \in A(K) $ and   $\mathbb{BC}\setminus( K \cup \mathcal{O})$  is of the type $\mathcal{T}_1$. So $F_j(Z_j)  \in A(K_j)$ for $ j = 1, 2 $ and  $\mathbb{C}(\mathbf{i})\setminus K_j $ for $ j = 1, 2 $ have finitely  many components. Then  by Mergelyan's approximation theorem for complex  function,  there exists a rational function $R_j(Z_j) $ for $ j = 1, 2 $ with poles in $P_j \in \mathbb{C}(\mathbf{i})\setminus K_j $  for $ j = 1, 2 $ such that,
	$$ \sup_{Z_j\in K_j} |F_j(Z_j) - R_j(Z_j)| < \epsilon~~for~~  j = 1, 2  .$$
Now,
\begin{eqnarray*}
	 \sup_{z\in K} |F(Z) - R(Z)|_{\mathbf{k}} &\preceq&  \sup_{Z_1\in K_1}|F_1(Z_1) - R_1(Z_1)| \mathbf{e_1} + \sup_{Z_2\in K_2}|F_2(Z_2) - R_2(Z_2) | \mathbf{e_2}\\ &\prec& \epsilon \mathbf{e_1}	+  \epsilon\mathbf{e_2} \\
	 &\prec&  \epsilon,
 \end{eqnarray*}
where $R(Z) = R_1(Z_1)\mathbf{e_1} + R_2(Z_2)\mathbf{e_2}$ is a bicomplex  rational function with bicomplex poles $ P = P_1 \mathbf{e_1}  + P_2\mathbf{e_2}  \in \mathcal{T}_1 $.

\item[(ii)]  We have $F(Z) = F_1(Z_1)\mathbf{e_1} + F_2(Z_2)\mathbf{e_2} \in A(K) $ and   $\mathbb{BC}\setminus( K \cup \mathcal{O})$  is of the type $\mathcal{T}_2$. So $F_1(Z_1)  \in A(K_1)$  and  $\mathbb{C}(\mathbf{i})\setminus K_1 $  have finitely  many components. Then  by Mergelyan's approximation theorem for complex  function,  there exist a rational function $R_1(Z_1) $  with poles in $P_1 \in (\mathbb{C}(\mathbf{i})\setminus K_1) $  such that,
		$$ \sup_{Z_1\in K_1} |F_1(Z_1) - R_1(Z_1)| < \epsilon .$$
Also, $F_2(Z_2)  \in A(K_2)$  and  $(\mathbb{C}(\mathbf{i})\setminus K_2) $  is connected, then again by Mergelyan's approximation theorem for complex  function,  there exist a polynomial $P_2(Z_2) $  with only pole in $\infty \in (\mathbb{C}(\mathbf{i})\setminus K_2) $  such that,
		$$ \sup_{Z_2\in K_2} |F_2(Z_2) - P_2(Z_2)| < \epsilon .$$
Now,
\begin{eqnarray*}
			\sup_{z\in K} |F(Z) - R(Z)|_{\mathbf{k}} &\preceq&  \sup_{Z_1\in K_1}|F_1(Z_1) - R_1(Z_1)| \mathbf{e_1} + \sup_{Z_2\in K_2}|F_2(Z_2) - P_2(Z_2) | \mathbf{e_2}\\ &\prec& \epsilon \mathbf{e_1}	+  \epsilon\mathbf{e_2} \\
			&\prec&  \epsilon,
\end{eqnarray*}
where $R(Z) = R_1(Z_1)\mathbf{e_1} + P_2(Z_2)\mathbf{e_2}$ is a bicomplex  rational function with bicomplex poles as bicomplex infinity $ \beta_1 \mathbf{e_1} + \infty\mathbf{e_2} \in \overline{\mathbb{BC}}\setminus( K \cup \mathcal{O}) $,  with $\beta_1 \in (\mathbb{C}(\mathbf{i})\setminus K_1)$.

\item[(iii)]  The proof of this part is  similar to that of case (ii).

\item[(iv)] We have $F(Z) = F_1(Z_1)\mathbf{e_1} + F_2(Z_2)\mathbf{e_2} \in A(K) $ and   $\mathbb{BC}\setminus( K \cup \mathcal{O})$  is of the type $\mathcal{T}_4$. So $F_j(Z_j)  \in A(K_j)$ for $ j = 1, 2 $ and  $\mathbb{C}(\mathbf{i})\setminus K_j $ for $ j = 1, 2 $ is connected. Then  by Mergelyan's approximation theorem for complex  function,  there exist a polynomials $P_j(Z_j) $ for $ j = 1, 2 $ with poles as  $\infty \in \mathbb{C}(\mathbf{i})\setminus K_j $  for $ j = 1, 2 $ such that,
$$ \sup_{Z_j\in K_j} |F_j(Z_j) - R_j(Z_j)| < \epsilon~~for~~  j = 1, 2  .$$
Now,
\begin{eqnarray*}
	\sup_{z\in K} |F(Z) - P(Z)|_{\mathbf{k}} &\preceq&  \sup_{Z_1\in K_1}|F_1(Z_1) - P_1(Z_1)| \mathbf{e_1} + \sup_{Z_2\in K_2}|F_2(Z_2) - P_2(Z_2) | \mathbf{e_2}\\ &\prec& \epsilon \mathbf{e_1}	+  \epsilon\mathbf{e_2} \\
	&\prec&  \epsilon,
\end{eqnarray*}
where $P(Z) = P_1(Z_1)\mathbf{e_1} + P_2(Z_2)\mathbf{e_2}$ is a bicomplex polynomial with bicomplex pole as bicomplex infnity $ \infty \mathbf{e_1} + \infty\mathbf{e_2} \in \overline{\mathbb{BC}}\setminus( K \cup \mathcal{O}) $. This complete proof.
\end{itemize}
\end{proof}


\begin{thebibliography}{99}

\bibitem{Alpa} Alpay D, Luna-Elizarraras ME, Shapiro M, Struppa DC. Basics of functional analysis with bicomplex scalars and bicomplex Schur analysis. Springer Briefs in Mathematics:2014.  

\bibitem{Aren} Arens A, Kelly JL, Characterization of spaces of continuous functions over a compact Hausdorff space. Trans. Am. Math. Soc.. 1947;62(3):499-508.

\bibitem{Arif1} Arif M, Kumar R. Runges approximation theorem for bicomplex. Pages-11. [preprint]. 

\bibitem{Arif2} Arif M, Ali A, Singh R, Kumar R. Bicomplex univalent function. Aust. J. Math. Anal. Appl. 2023;20(1):Art.3.

\bibitem{Bory} Bory-Reyes J, Perez-Regalado CO, Shapiro M. Cauchy type integral in bicomplex setting and its properties. Complex Analysis and Operator Theory. 2019;13(6):2541-2573.

\bibitem{Cato2} Catoni F, Cannata R, Catoni V, Zampetti P. Two-dimensional hypercomplex numbers and related trigonometries and geometries. Adv. Appl. Clifford Algebras. 2004;14(1):47-68.

\bibitem{Cato3} Catoni F, Boccaletti D, Cannata R, Catoni V, Nichelatti E, Zampatti P. The mathematics of Minkowski space-time. Birkhauser;Basel:2008.

\bibitem{Colo1} Colombo F, Sabadin I, Struppa DC, Vajiac A, Vajiac MB. Singularities of functions of one and several bicomplex variables. Ark. Mat..2011;49(2):277-294.

\bibitem{Colo2} Colombo F, Sabadini I, Struppa DC. Bicomplex holomorphic functional calculus. Math. Nachr.. 2013;287(13):1093-1105.

\bibitem{Dave} Davenport CM, Ablamowicz R,  Parra JM, Josep M, Lounesto P. (eds.) A commutative hypercomplex algebra with associated function theory. Clifford Algebras with Numeric and Symbolic Computations. Birkhauser;Boston:1996.

\bibitem{ghos} Ghosh C. Bicomplex Mobius transformation. 2017;ArXiv preprint arXiv:1706.07699.

\bibitem{Green} Green E, Krantz G. Function theory of one complex variable. third edition. American Mathematical Society Providence. Rhode Island; Graduate Studies in Mathematics Volume 40:2006.

\bibitem{Kisi} Kisil VV.  Starting with the group SL2(R), Notices Amer. Math. Soc.. 2007;54:1458-146.

\bibitem{Luna3} Luna-Elizarraras ME, Perez-Regalado CO, Shapiro M. On linear functionals and Hahn-Banach theorems for hyperbolic and bicomplex modules. Adv. Appl. Clifford Algebr.. 2014;24(4):1105-1129.

\bibitem{Luna5} Luna-Elizarraras ME, Shapiro M, Struppa DC. On Clifford analysis for holomorphic mappings. Adv. Geom.. 2014;14(3):413-426.

\bibitem{Luna6} Luna-Elizarraras ME, Shapiro M, Struppa DC, Vajiac A.  Bicomplex holomorphic functions: The algebra geometry and analysis of bicomplex numbers. Frontiers in Mathematics. Springer;New York:2015.

\bibitem{Luna7} Luna-Elizarraras ME, Perez-Regalado CO, Shapiro M. On the bicomplex Gleason-Kahane Zelazko Theorem. Complex Anal. Oper. Theory. 2016;10(2):327-352. 

\bibitem{Luna8} Luna-Elizarraras ME, Perez-Regalado CO, Shapiro M. On the Laurent series for bicomplex holomorphic functions. Complex variables and elliptic equations. 2017;62(9):1266-1286.

\bibitem{Luna9} Luna-Elizarraras ME, Shapiro M, Balanhin A. Fractal-type sets in the four-dimensional space using bicomplex and hyperbolic numbers. Analysis and Mathematical Physics. 2020;10(1):1-30.	

\bibitem{Luna11} Luna-Elizarraras ME, Perez-Regalado CO, Shapiro M. Singularities of bicomplex holomorphic functions. Mathematical Methods in the Applied Sciences:2021.

\bibitem{Luna10} Luna-Elizarrarás ME. Integration of functions of a hyperbolic variable. Complex Analysis and Operator Theory. 2022;16(3):1-21.

\bibitem{Pric} Price GB.  An introduction to multicomplex spaces and functions. 3rd Edition. Marcel Dekker;New York:1991.

\bibitem{Reye} Reyes JB, Regalado CO, Shapiro M. Cauchy type integral in bicomplex setting and its properties. Comp. Analysis and Oper. Thoery. 2019;13(6):2541-2573.

\bibitem{Rile}  Riley JD.  Contributions to the theory of functions of a bicomplex variable. Tohoku Math. J. Second Ser.. 1953;5(2):132-165.

\bibitem{Roch1} Rochon D, Shapiro M. On algebraic properties of bicomplex and hyperbolic numbers. Anal. Univ. Oradea, fasc. math. 2004;11(71):110.

\bibitem{Rudi2} Rudin W.  Real and Complex Analysis. 3nd edition. McGraw Hill;New York:1987.

\bibitem{Sain} Saini H, Sharma A, Kumar R. Some fundamental theorems of functional Analysis with bicomplex and hyperbolic scalars.  Adv. Appl. Clifford Algebras. 2020;30(5):1-23.

\bibitem{Stef} Stefan R. Bicomplex algebra and function theory. 2001;arXiv preprint math/0101200.


\end{thebibliography}
\end{document}